\documentclass[10pt,reqno]{amsart}
\usepackage{amssymb,mathrsfs,graphicx}
\usepackage{ifthen}
\usepackage[margin=1.1in]{geometry}
\usepackage{colortbl}
\definecolor{black}{rgb}{0.0, 0.0, 0.0}
\definecolor{red}{rgb}{1.0, 0.5, 0.5}
\provideboolean{shownotes} 
\setboolean{shownotes}{true} 
\newcommand{\margnote}[1]{
\ifthenelse{\boolean{shownotes}}%
{\marginpar{\raggedright\tiny\texttt{#1}}}%
{}%
}
\newcommand{\hole}[1]{
\ifthenelse{\boolean{shownotes}}%
{\begin{center} \fbox{ \rule {.25cm}{0cm} \rule[-.1cm]{0cm}{.4cm}
\parbox{.85\textwidth}{\begin{center} \texttt{#1}\end{center}} \rule
{.25cm}{0cm}}\end{center}} {} }

\title[Compressible Euler equations with a nonlocal dissipation]{The global Cauchy problem for compressible Euler equations with a nonlocal dissipation}

\author[Choi]{Young-Pil Choi}
\address[Young-Pil Choi]{\newline Department of Mathematics and Institute of Applied Mathematics\newline
Inha University, Incheon 402-751, Republic of Korea}
\email{ypchoi@inha.ac.kr}

\numberwithin{equation}{section}

\newtheorem{theorem}{Theorem}[section]
\newtheorem{lemma}{Lemma}[section]

\newtheorem{proposition}{Proposition}[section]
\newtheorem{remark}{Remark}[section]

\newcommand{\R}{\mathbb R}
\newcommand{\ls}{\lesssim}

\newcommand{\T}{\mathbb T}

\newcommand{\mc}{\mathcal C}

\newcommand{\bq}{\begin{equation}}
\newcommand{\eq}{\end{equation}}
\newcommand{\e}{\varepsilon}
\newcommand{\lt}{\left}
\newcommand{\rt}{\right}

\newcommand{\pa}{\partial}

\def\charf {\mbox{{\text 1}\kern-.30em {\text l}}}

\begin{document}
\allowdisplaybreaks

\date{\today}

\subjclass[]{}
\keywords{Global existence, Euler alignment system, nonlocal dissipation, large-time behavior.}


\begin{abstract}This paper studies the global existence and uniqueness of strong solutions and its large-time behavior for the compressible isothermal Euler equations with a nonlocal dissipation. The system is rigorously derived from the kinetic Cucker-Smale flocking equation with strong local alignment forces and diffusions through the hydrodynamic limit based on the relative entropy argument. In a perturbation framework, we establish the global existence of a unique strong solution for the system under suitable smallness and regularity assumptions on the initial data. We also provide the large-time behavior of solutions showing the fluid density and the velocity converge to its averages exponentially fast as time goes to infinity. 
\end{abstract}

\maketitle \centerline{\date}


%
%
%
%
\section{Introduction}
In this paper, we are interested in the global Cauchy problem for compressible Euler equations with a nonlocal dissipation in the periodic domain $\T^d$ with $d \geq 1$. More precisely, we study the global existence of a unique strong solution and the large-time behavior for the following compressible isothermal Euler equations with a nonlocal dissipation:
\begin{align}\label{CEN}
\begin{aligned}
&\partial_t \rho + \nabla_x \cdot (\rho u) = 0, \quad x \in \T^d, \quad t >0,\cr
&\partial_t (\rho u) + \nabla_x \cdot (\rho u \otimes u) + \nabla_x \rho = -\rho\int_{\T^d} \psi(x-y)(u(x) - u(y))\rho(y) \,dy,
\end{aligned}
\end{align}
with initial data 
\bq\label{ini-CEN}
(\rho,u)|_{t=0} =: (\rho_0,u_0), \quad x \in \T^d.
\eq
Here $\rho$ and $u$ are the density and velocity of the flow, respectively, and $\psi$ represents a communication weight, which gives the local averaging measuring the alignment in velocity among individuals. Throughout this paper, we assume that  $\psi$ satisfies
\bq\label{ass_psi}
\psi \in W^{s+1,\infty}(\T^d) \mbox{ for some } s \in \R_+, \quad \psi(x) = \psi(-x), \quad \mbox{and} \quad \psi \geq \psi_m > 0.
\eq
We also may assume, without loss of generality, that $\rho$ is a probability density function, i.e., $\|\rho(\cdot,t)\|_{L^1(\T^d)} = 1$ for $t\geq 0$ since the total mass is conserved in time. Note that the system \eqref{CEN} reduces to the standard isothermal Euler system provided $\psi \equiv 0$.

In the context of multi-agents interactions, the system \eqref{CEN} arises as macroscopic descriptions for the following Newton type microscopic model for interacting many-body system exhibiting a flocking phenomenon \cite{CFRT10,CS07,HL09,HT08}:
\begin{align}\label{CS}
\begin{aligned}
\frac{d x_i(t)}{dt} &= v_i(t), \quad i = 1,\cdots, N, \quad t > 0,\cr
\frac{d v_i(t)}{dt} &= \frac1N \sum_{j=1}^N \psi(x_i(t) - x_j(t))(v_j(t) - v_i(t)).
\end{aligned}
\end{align}
From the particle system \eqref{CS}, we can derive a kinetic equation, mesoscopic descriptions, by using BBGKY hierarchies or mean-field limits \cite{CCR11,HL09,HT08} when the number of individuals goes to infinity, i.e., $N \to \infty$. To be more precise, the mesoscopic observables for the system \eqref{CS} can be estimated from the velocity moments of the density function $f$, which solves the following Vlasov type equation:
\bq\label{Vlasov}
\pa_t f + v \cdot \nabla_x f + \nabla_v \cdot \lt(F[f]f\rt) = 0, 
\eq
where $F[f] := F[f](x,v,t)$ is the velocity alignment force field given by
\[
F[f] = \int \psi(x-y)(w-v)f(y,w,t)\,dydw.
\]
The existence of weak and strong solutions, and the rigorous derivation of the kinetic equation \eqref{Vlasov} are well studied in \cite{CCR11,HL09,HT08}. We also refer to \cite{BCC11,CCH14,CCH14_2,CCHSapp,CS18} for the local/glocal-in-time existence theories and the mean-field limit for the swarming models with singular interaction kernels.

At the formal level, in order to derive the hydrodynamic equations of the form \eqref{CEN}, we can take into account the moments on the kinetic equations:
\[
\rho = \int f\,dv \quad \mbox{and} \quad \rho u = \int vf\,dv.
\]
Then we can easily check that the local density $\rho$ and velocity $u$ satisfy
\begin{align}\label{hydro}
\begin{aligned}
&\pa_t \rho + \nabla_x \cdot (\rho u) = 0,\cr
&\pa_t (\rho u) + \nabla_x \cdot (\rho u \otimes u) + \nabla_x \cdot \lt(\int (v-u)\otimes (v-u)f(x,v,t)\,dv\rt) \cr
&\hspace{2.3cm} =-\rho\int \psi(x-y)(u(x) - u(y))\rho(y) \,dy.
\end{aligned}
\end{align}
Note that the momentum equations in \eqref{hydro} are not closed, and up to now, suitable closure conditions for \eqref{hydro} are not known. However, to close the above system, we can formally consider the mono-kinetic ansatz for $f$:
\bq\label{mono}
f(x,v,t) \simeq \rho \delta_{v - u(x,t)}
\eq
or the local Maxwell type ansatz for $f$:
\bq\label{maxwell}
f(x,v,t) \simeq \rho(x,t) \exp\lt(-\frac{|v - u(x,t)|^2}{2} \rt).
\eq
These formal assumptions on $f$ give the Euler alignment system; the system \eqref{hydro} is reduced to the system \eqref{CEN} without pressure by assuming the mono-kinetic ansatz \eqref{mono}, the local Maxwell type ansatz \eqref{maxwell} gives our main system \eqref{CEN}.  Very recently, it is obtained that the rigorous derivations of the pressureless Euler alignment system from the equation \eqref{Vlasov} in \cite{FKpre}. More specifically, by considering the strong local alignment forces, for instance $(1/\e)\nabla_v \cdot \lt((v - u)f\rt)$ on the right side of the equation \eqref{Vlasov} and studying asymptotic limit limit $\e \to 0$, the convergence $f_\e(x,v,t) \rightharpoonup \rho(x,t) \delta_{v - u(x,t)}$ in the sense of distributions is found in \cite{FKpre}. The derivation of our main system \eqref{CEN} is established in \cite{KMT15} by taking into account the strong local alignment forces and diffusions. We give more details on that in Section \ref{sec_rigo} since it is very closely related to our current work. We refer to the recent reviews \cite{CCP17,CHL17} for the detailed descriptions of the modeling and related literature.

There are several works on the pressureless Euler alignment system; the global regularity of classical solutions is obtained in the Eulerian formulation \cite{HKK14} and in the Lagrangian formulation \cite{CHpre,HKK15}, and critical thresholds between the supercritical regions with finite-time breakdown and the subcritical region with global-in-time regularity of classical solutions are investigated in one dimension \cite{CCTT16,CCZ16,TT14}. More recently, the global regularity for the pressureless fractional Euler alignment system is also established in \cite{DKRTpre,KTpre,ST17,STapp}. Despite those fruitful developments on the existence theory and blow-up analysis for the pressureless Euler type system, to the best knowledge of the author, the global existence and the large-time behavior of strong solutions of the system \eqref{CEN}, i.e., the isothermal Euler alignment system, have not been studied so far. Note that the local-in-time existence and uniqueness of strong solutions are founded in \cite{KMT15}, and the global regularity based on the estimate of Riemann invariants is obtained in \cite{CCTT16} for the system \eqref{CEN} with the constant communication weights, $\psi \equiv 1$ for instance. It is worth mentioning that the presence of pressure destroys the characteristic structure, thus more careful analysis is required.

The main objective of this paper is to establish the global existence and the large-time behavior of solutions to the system \eqref{CEN}. The main difficulty is of course in obtaining an appropriate dissipative effect from the interaction term on the right hand side of the momentum equation in \eqref{CEN} and ruling out the possibility of the formation of singularities in a finite time. Note that it is well known that solutions of compressible Euler equations even with smooth initial conditions can develop a shock in a finite time, see \cite{Chen05} and the references therein for general survey of the Euler equations. Inspired by recent works \cite{Choi16,CK16}, where two-phase fluid models are studied, together with our careful analysis on the nonlocal interaction term, we show that the finite time blow up of strong solutions can be prevented under suitable smallness and regular assumptions on the initial data. We also want to emphasize that our strategy is a bit different from the one proposed in \cite{STW} for the compressible Euler equations with a linear damping. We do not estimate the time derivative of solutions in the desired Sobolev space. Instead of that, we consider a type of crossing term, which clearly gives the dissipation rate for the density and makes the estimates simpler, see Lemma \ref{lem:aux2} for more details. For the large-time behavior estimate, we employ a Lyapunov function approach based on $L^2$-estimate due to the presence of pressure. For the pressureless Euler alignment system, i.e., the system \eqref{CEN} without pressure, the large-time behavior of solutions is well studied in \cite{CHL17,CHpre,TT14} by using the method of characteristics based on $L^\infty$-estimate. More precisely, if there is no pressure in \eqref{CEN}, we can show that
\[
\sup_{x,y \in supp(\rho(\cdot,t))}|u(x,t) - u(y,t)| \to 0 \quad \mbox{as} \quad t \to \infty,
\]
under certain assumptions on the initial data and the weight function $\psi$. However, as mentioned above, it is more delicate to trace the dynamics of the system \eqref{CEN} along the characteristics due to the presence of pressure. This is the main reason why we employ the Lyapunov function approach to the estimate of large-time behavior of solutions. Furthermore, unlike the pressureless Euler alignment system, see \cite{HKK14,HKK15} for instance, a fluctuated energy $\int_{\T^d} \rho |u - m_c|^2\,dx$ with $m_c := \int_{\T^d} \rho u\,dx$ is not dissipative because of the pressure, thus further estimates to obtain the dissipation rate of the density are required. To handle this problem, we take the Bogovskii type estimate used in \cite{Choi16,Choi16_2,CK16} based on the elliptic regularity for Poisson's equation. This allows us to obtain a Gronwall's inequality for the proposed Lyapunov function, and we have the time-asymptotic velocity alignment behavior of solutions. Our strategy for the estimate of large-time behavior of solutions does not require any smallness assumption on the solutions, see Theorem \ref{thm:large} for details.

Here we introduce several notations used throughout the paper. For a function $u= u(x)$, $\|u\|_{L^p}$ denotes the usual $L^p(\T^d)$-norm. We also denote by $C$ a generic positive constant independent of $t$ and $f \ls g$ represents that there exists a positive constant $C>0$ such that $f \leq C g$ For simplicity, we often drop $x$-dependence of a differential operator $\nabla_x$, that is, $\nabla u := \nabla_x u$. For any nonnegative integer $k$, $H^k$ denote the $k$-th order $L^2$ Sobolev space. $\mc^k([0,T]:E)$ is the set of $k$-times continuously differentiable functions from an interval $[0,T] \subset \R$ into a Banach space $E$, and $L^p(0,T;E)$ is the set of the $L^p$ functions from an interval $(0,T)$ to a Banach space $E$. $\nabla^k$ denotes any partial derivative $\pa^\alpha$ with multi-index $\alpha$, $|\alpha| = k$.

\subsection{Rigorous derivation of \eqref{CEN} from a kinetic equation}\label{sec_rigo}
In \cite{KMT15}, the hydrodynamic limit of a kinetic Cucker-Smale flocking model with strong local alignment force and diffusion is investigated. More precisely, let $f(x,v,t)$ be the particle distribution function at $(x,v) \in \T^d \times \R^d$ at time $t$. Then the hydrodynamic limit $\e \to 0$ in the following kinetic equation is studied in \cite{KMT15}:
\bq\label{eq_kin}
\pa_t f^\e + v \cdot \nabla_x f^\e + \nabla_v \cdot (F[f^\e]f^\e) = \frac1\e \nabla_v \cdot ((v - u^\e)f^\e + \nabla_v f^\e),
\eq
where the alignment force $F[f^\e]$ and the local velocity $u^\e$ are given by
\[
F[f^\e](x,v,t) := \int_{\T^d \times \R^d} \psi(x-y)(w-v) f^\e(y,w,t)\,dydw \quad \mbox{and} \quad u^\e := \frac{\int_{\R^d} vf^\e\,dv}{\int_{\R^d} f^\e\,dv},
\]
respectively. At the formal level, we can easily find 
\[
f^\e \to f = \rho(x,t) e^{-\frac{|v-u(x,t)|^2}{2}} \quad \mbox{as} \quad \e \to 0,
\]
by setting $\e = 0$ in \eqref{eq_kin}, and this indicates that the limiting system should be our main system \eqref{CEN}. In \cite{KMT15}, by employing the relative entropy argument, the following inequality is obtained:
\begin{align}\label{ineq_entropy}
\begin{aligned}
& \int_{\T^d} \frac{\rho^\e}{2}|u^\e - u|^2\,dx + \int_{\T^d}\int_\rho^{\rho^\e} \frac{\rho^\e - z}{z} \,dzdx  \cr
&\qquad + \frac12\int_0^{T^*} \int_{\T^d \times \T^d} \psi(x-y)\rho^\e(x) \rho^\e(y)\lt((u^\e(x) - u(x)) - (u^\e(y) - u(y)) \rt)^2 dxdydt\cr
&\qquad \quad \leq C\sqrt{\e},
\end{aligned}
\end{align}
for well-prepared initial data, where $\rho^\e = \int_{\R^d} f^\e \,dv$, $\rho^\e u^\e = \int_{\R^d} vf^\e\,dv$, and $T^*>0$ is the maximal time for which there exists a strong solution to the system \eqref{CEN}. This gives the following strong converges of weak solutions of the equation \eqref{eq_kin} to the strong solutions of the isothermal Euler alignment system \eqref{CEN}: 
\begin{align}\label{strong_conv}
\begin{aligned}
&f^\e \to \rho e^{-\frac{|u-v|^2}{2}} \quad \mbox{in} \quad L^1_{loc}(0,T^*; L^1(\T^d \times \R^d)) \quad \mbox{as} \quad \e \to 0,\cr
&\rho^\e \to \rho, \quad \rho^\e u^\e \to \rho u, \quad \mbox{and} \quad \rho^\e |u^\e|^2 \to \rho u^2 \quad \mbox{in} \quad L^1_{loc}(0,T^*; L^1(\T^d)) \quad \mbox{as} \quad \e \to 0.
\end{aligned}
\end{align}
As mentioned above, the above hydrodynamic limit $\e \to 0$ holds as long as there exists a strong solution to the limiting system, which is our main system \eqref{CEN}, since the global existence of weak solutions is established in \cite{KMT13}. Note that the local-in-time unique strong solution is obtained in \cite{KMT15}, and it is not obvious to obtain the global regularity of strong solutions to the system \eqref{CEN} due to the nonlocal nonlinear external forces. Thus combining results in \cite{KMT13, KMT15} only yields that the inequality \eqref{ineq_entropy} holds for a short time $T^* > 0$, which is given as the above. 

\subsection{Main results} In this part, we state our main results on the global-in-time existence and uniqueness of strong solutions and its large-time behavior for the system \eqref{CEN}.  For the global-in-time regularity of solutions, we reformulate the system \eqref{CEN}, by setting $h: = \ln \rho$, as 
\begin{align}\label{main_eq2}
\begin{aligned}
&\pa_t h + \nabla h \cdot u + \nabla \cdot u = 0, \quad x \in \T^d, \quad t > 0,\cr
&\pa_t u + u \cdot \nabla u + \nabla h = - \int_{\T^d} \psi(x-y)(u(x) - u(y)) e^{h(y)}\,dy,
\end{aligned}
\end{align}
with the initial data
\bq\label{ini-main_eq2}
(h(x,t),u(x,t))|_{t=0} =: (h_0(x) = \ln\rho_0(x),u_0(x)), \quad x \in \T^d.
\eq
Then we present our first result on the existence theory for the reformulated system \eqref{main_eq2}.
\begin{theorem}\label{thm:ext}Let $s> d/2+1$. Suppose that the initial data $(h_0,u_0)$ satisfy
\[
(h_0,u_0) \in H^s(\T^d) \times H^s(\T^d),
\]
and the communication weight function $\psi$ satisfies \eqref{ass_psi}. If $\|(h_0,u_0)\|_{H^s} \leq \e_1$ with sufficiently small $\e_1>0$, then the Cauchy problem \eqref{main_eq2}-\eqref{ini-main_eq2} has a unique global strong solution $(h,u) \in \mc([0,\infty);H^s(\T^d)) \times \mc([0,\infty);H^s(\T^d))$.
\end{theorem}
\begin{remark}From the structure of the system \eqref{main_eq2}, we can easily find 
\[
h,u \in \mc([0,\infty);H^s(\T^d)) \cap \mc^1([0,\infty);H^{s-1}(\T^d)).
\]
Moreover, it follows from Lemma \ref{lem_equiv1} in Section \ref{sec_ext} that Theorem \ref{thm:ext} implies the global-in-time existence and uniqueness of classical solutions to the system \eqref{CEN}.
\end{remark}
\begin{remark}As mentioned in Section \ref{sec_rigo}, our main system \eqref{CEN} can be rigorously derived from the kinetic equation \eqref{eq_kin} through the hydrodynamic limit $\e \to 0$ and this limit holds as long as there exists a unique strong solution to the system \eqref{CEN}. Thus, combining the global-in-time existence result in Theorem \ref{thm:ext} and the previous result on the hydrodynamic limit in \cite{KMT15} yields that the relative entropy inequality \eqref{ineq_entropy} holds for all time, and subsequently, the strong convergences \eqref{strong_conv} also hold for all time.
\end{remark}

In order to present our second result of the current work on the large-time behavior of classical solutions showing the velocity alignment behavior, we introduce a Lyapunov function measuring the fluctuation of momentum and mass from the corresponding averaged quantities:
\[
\mathcal{F}(t) := \int_{\T^d} \rho|u - m_c|^2 \,dx + \int_{\T^d} (\rho - 1)^2 \,dx \quad \mbox{where} \quad m_c(t) := \int_{\T^d} \rho u\,dx.
\]

\begin{theorem}\label{thm:large}Let $(\rho ,u)$ be any global classical solutions to the system \eqref{CEN}-\eqref{ini-CEN} satisfying $(\rho,u) \in L^\infty(\T^d \times \R_+)$. Suppose that the communication weight function $\psi$ satisfies \eqref{ass_psi} and the norm of initial momentum $m_c(0)$ is small enough. Then we have
\[
 \mathcal{F}_0 e^{-\tilde ct} \ls \mathcal{F}(t) \ls \mathcal{F}_0 e^{-ct}, \quad t \geq 0,
\]
where $c$ and $\tilde c$ are positive constants independent of $t$.
\end{theorem}
\begin{remark}For the estimate of large-time behavior, we do not require any smallness assumptions on solutions, we only need small initial momentum. Note that the momentum is conserved in time, see Lemma \ref{lem:1} below. Thus if we assume that $m_c(0) = \int_{\T^d} \rho_0 u_0\,dx = 0$, then $m_c(t) = 0$ for all $t \geq 0$, and this gives $\mathcal{F}(t) = E(t)$ for all $t \geq 0$, where $E(t) := \int_{\T^d} \rho|u|^2\,dx + \int_{\T^d} (\rho - 1)^2\,dx$. In general, we obtain that the Lyapunov function $\mathcal{F}$ is bounded from above by the total energy $E$:
\[
\mathcal{F}(t) = \int_{\T^d} \rho|u|^2\,dx + \int_{\T^d} (\rho - 1)^2 \,dx - |m_c|^2 \leq E(t).
\]
\end{remark}
\begin{remark}
The classical solution obtained in Theorem \ref{thm:ext} automatically satisfies the assumptions in Theorem \ref{thm:large}.
\end{remark}

\subsection{Outline of the paper} The rest of this paper is organized as follows. In Section \ref{sec_pre}, we provide a priori energy estimates and recall several useful estimates, which will be significantly used later. In Section \ref{sec_ext}, we study the local-in-time existence and uniqueness of strong solutions to the system \eqref{main_eq2} and the equivalence relation between the system \eqref{CEN} and the reformulated system \eqref{main_eq2}. We then provide the a priori estimates of solutions in the desired Sobolev spaces, which enables us to extend the local-in-time solution to the global one. Finally, Section \ref{sec_lt} is devoted to investigate the large-time behavior of classical solutions for the system \eqref{CEN}.

%
%
%
%
\section{Preliminaries}\label{sec_pre}
\subsection{Energy estimates} In this part, we provide the conservation of momentum and the energy estimate. 
\begin{lemma}\label{lem:1}Let $(\rho ,u)$ be a global classical solution to the system \eqref{CEN}-\eqref{ini-CEN}. Then we have
$$\begin{aligned}
&\frac{d}{dt} \int_{\T^d} \rho u \,dx = 0,\cr
&\frac{d}{dt}\lt( \frac12\int_{\T^d} \rho|u|^2 dx + \int_{\T^d} \rho\ln\rho\,dx\rt) + \int_{\T^d \times\T^d} \psi(x-y)|u(x)-u(y)|^2 \rho(x)\rho(y) \,dxdy= 0.
\end{aligned}$$
\end{lemma}
\begin{proof} The proof can be easily obtained by using the symmetry of the communication weight function $\psi$.
\end{proof}
\begin{remark}For the reformulated system, we have the following identities:
\[
\int_{\T^d} e^{h(x,t)}\,dx = \int_{\T^d} e^{h_0(x)}\,dx \quad \mbox{and} \quad \int_{\T^d} e^{h(x,t)} u(x,t) \,dx = \int_{\T^d} e^{h_0(x)} u_0(x) \,dx.
\]
for all $t \geq 0$.
\end{remark}
In the following two lemmas, we give a relationship between $\int_{\T^d} \rho \ln \rho\,dx$ and $\|\rho-1\|_{L^2}^2$.
\begin{lemma}\label{lem:2} Let $\rho \in \mc^1(\T^d \times [0,T])$. Then we have
\[
\frac{d}{dt}\int_{\T^d} \rho \ln \rho \,dx = \frac{d}{dt}\int_{\T^d} \rho \int_1^\rho \frac{z - 1}{z^2} \,dzdx.
\]
\end{lemma}
\begin{proof}A straightforward computation yields the result.
\end{proof}
\begin{lemma}\label{lem:3} Let $ \rho \in [0,\bar\rho]$ with $\bar\rho > 0$. Then, there exist positive constants $c_1, c_2 > 0$, we have
\[
c_1(\rho - 1)^2 \leq  \rho\int_1^\rho \frac{z - 1}{z^2}\,dz \leq c_2(\rho - 1)^2. 
\]
\end{lemma}
\begin{proof} Set
\[
g(\rho) := \frac{\displaystyle \rho\int_1^\rho \frac{z - 1}{z^2}\,dz}{(\rho - 1)^2}.
\]
Since 
\[
\rho\int_1^\rho \frac{z - 1}{z^2}\,dz = \rho \ln \rho + 1 - \rho,
\]
we easily find 
\[
\lim_{\rho \to 0} g(\rho) = 1>0 \quad \mbox{and} \quad \lim_{\rho \to 1}g(\rho) = \frac12>0.
\]
Thus we deduce that $g(\rho)$ is a continuous function on $[0,\bar\rho]$ with $g(\rho) > 0$, and this concludes the desired result.
\end{proof}
Summarizing the above discussions, we have the following energy estimate.
\begin{lemma}Let $(\rho ,u)$ be any global classical solutions to the system \eqref{CEN}-\eqref{ini-CEN} satisfying $\rho \in [0,\bar\rho]$ with $\bar \rho > 0$. Then we have
\begin{align*}
\begin{aligned}
&\int_{\T^d} \rho|u|^2 \,dx + \int_{\T^d} (\rho - 1)^2 \,dx + \int_0^t\int_{\T^d \times \T^d} \psi(x-y)|u(x)-u(y)|^2 \rho(x)\rho(y) \,dxdyds\cr
&\qquad \ls \int_{\T^d} \rho_0|u_0|^2 \,dx + \int_{\T^d} (\rho_0 - 1)^2 \,dx.
\end{aligned}
\end{align*}
We set 
\[
E(t):= \int_{\T^d} \rho|u|^2 dx + \int_{\T^d} (\rho - 1)^2 dx,
\]
then we find
\[
E(t) \ls E_0.
\]
\end{lemma}
Even though we find a good energy estimate for the system \eqref{CEN}, we cannot directly employ that estimate for the reformulated system \eqref{main_eq2}. For that, we need the following auxiliary lemma whose proof can be found in \cite{Choi15}.
\begin{lemma} For $0 < a \leq f(x) \leq b$ with $a \leq 1 \leq b$, there exist constants $C(a), C(b) > 0$ such that 
\[
C(b)\int_{\T^d} (f-1)^2\,dx \leq \int_{\T^d} (\ln f)^2 \,dx \leq C(a)\int_{\T^d} (f-1)^2\,dx,
\]
where $C(a)$ and $C(b)$ are explicitly given by
\[
C(a) := \max\lt\{1, \lt(\frac{\ln a}{1-a}\rt)^2\rt\} \quad \mbox{and} \quad C(b) := \min\lt\{1, \lt(\frac{\ln b}{b-1}\rt)^2\rt\},
\]
respectively. 
\end{lemma}
Using the above lemma, we finally have the $L^2$-energy estimate for the reformulated system \eqref{main_eq2}.
\begin{lemma}\label{lem_energy_hu} Let $(h ,u)$ be any global classical solutions to the system \eqref{main_eq2}-\eqref{ini-main_eq2} satisfying $h \in L^\infty(\T^d \times \R_+)$. Then we have
$$\begin{aligned}
&\int_{\T^d} e^{h(x,t)}|u(x,t)|^2 \,dx + \int_{\T^d} h(x,t)^2 \,dx + \int_0^t\int_{\T^d \times \T^d} \psi(x-y)|u(x,s)-u(y,s)|^2 e^{h(x,s)+h(y,s)} \,dxdyds\cr
&\qquad \ls \int_{\T^d} e^{h_0(x)}|u_0(x)|^2 \,dx + \int_{\T^d} h_0(x)^2 \,dx.
\end{aligned}$$
\end{lemma}

\subsection{Moser type inequality \& Bogovskii type estimate}\label{sec_bogo}
We first recall the Moser type inequality which will be frequently used in this paper later for the global-in-time regularity of solutions.
\begin{lemma}\label{lem:bern}For any pair of functions $f,g \in (H^k \cap L^\infty)(\T^d)$, we have
\[
\|\nabla^k (fg)\|_{L^2} \ls \|\nabla^k f\|_{L^2}\|g\|_{L^\infty} + \|f\|_{L^\infty}\|\nabla^k g\|_{L^2}.
\]
Furthermore if $\nabla f \in L^\infty(\T^d)$ we have
\[
\|\nabla^k (fg) - f\nabla^k g\|_{L^2} \ls \|\nabla^k f\|_{L^2}\|g\|_{L^\infty} + \|\nabla f\|_{L^\infty}\|\nabla^{k-1} g\|_{L^2}.
\]
\end{lemma}
For the large-time behavior estimate, in order to have the dissipation rate for the density, we recall Bogovskii type estimate for the following stationary transport equation with auxiliary equations:
\begin{equation} \label{B-1}
\nabla_x \cdot v = f, \quad \nabla_x \times v = 0 \quad \mbox{in}~~\T^d, \quad \mbox{and} \quad \int_{\T^d} \,v\,dx = 0.
\end{equation}
For a given function $f \in L^2_\#(\T^d) := \{ f \in L^2(\T^d) : \int_{\T^d} f\,dx = 0 \}$, consider a operator ${\mathcal B}: f \mapsto v$. Then, the relations between the norms of $v$ and $f$ can be given in the following lemma, which can be obtained from the estimate of elliptic regularity for Poisson's equation, see \cite{Choi16_2,Galdi} for more details.
\begin{lemma}
\label{oper-b-lem}
Consider the equations \eqref{B-1} and the operator ${\mathcal B}$ defined above.
\begin{enumerate}
\item
$v = \mathcal{B}[f]$ is a solution to the problem \eqref{B-1} and a linear operator from $L^2_\#(\T^d)$ into $H^1(\T^d)$, i.e.,
\[
\|\mathcal{B}[f] \|_{H^1} \leq C\| f \|_{L^2},
\]
for some $C>0$.
\item If a function $f \in H^1(\T^d)$ can be written in the form
$f = \nabla \cdot g$ with $g \in H^1(\T^d)$, then
\[
\| \mathcal{B}[f] \|_{L^2} \leq C\|g\|_{L^2},
\]
for some $C>0$.
\end{enumerate}
\end{lemma}

%
%
%
%
\section{Global existence of strong solutions}\label{sec_ext}
\subsection{Local-in-time existence and uniqueness} In this subsection, we present the local existence of the unique strong solution for the system \eqref{CEN}. Note that as we did at the formal level, we can find the relation between the classical solutions $(\rho,u)$ and $(h,u)$ to the systems \eqref{CEN} and \eqref{main_eq2}, respectively, in the following two lemmas. The proofs can be obtained by taking the similar strategy as in \cite{STW}, where the compressible Euler equations with a linear damping is studied.
\begin{lemma}\label{lem_equiv1}
For any $T>0$, if $(\rho,u) \in \mathcal{C}^1(\T^d \times [0,T])$ is a solution to the system \eqref{CEN} with $\rho > 0$, then $(h,u) \in \mathcal{C}^1(\T^d \times [0,T])$ is a solution for the system \eqref{main_eq2} with $e^h > 0$. Conversely, if $(h,u) \in \mathcal{C}^1(\T^d \times [0,T])$ is a solution of the system \eqref{main_eq2} with $e^h > 0$, then $(\rho, u) \in \mathcal{C}^1(\T^d \times [0,T])$ is a solution of \eqref{CEN} with $\rho =  e^h > 0$.
\end{lemma}
\begin{lemma} For any $T>0$, if $(\rho,u) \in \mathcal{C}^1(\T^d \times [0,T])$ is a uniformly bounded solution of \eqref{CEN} with $\rho_0(x) > 0$, then $\rho(x,t) > 0$ on $\T^d \times [0,T]$. Conversely, If $(h,u)\in \mathcal{C}^1(\T^d \times [0,T])$ is a uniformly bounded solution of \eqref{main_eq2} with $e^{h_0} >0$, then $e^{h(x,t)} >0$ on $\T^d \times [0,T]$.
\end{lemma}
We can also find the following local-in-time existence theory for the system \eqref{main_eq2} by using a similar argument as in \cite{KMT15,Majda}.
\begin{lemma}\label{lem:local} Let $s > d/2+1$. If $(h_0,u_0) \in H^s(\T^d)$, then there exists a unique local solution $(h,u) \in \mathcal{C}([0,T]; H^s(\T^d)) \cap \mathcal{C}^1([0,T];H^{s-1}(
\T^d))$ to the system \eqref{main_eq2}-\eqref{ini-main_eq2} for some finite $T > 0$. More precisely, we can show that if $\|(h_0,u_0)\|_{H^s}^2 \leq M_1$, then there exist $T_0$ and $M_2 > M_1$ such that $\sup_{0\leq t \leq T_0}\|(h,u)\|_{H^s}^2 \leq M_2$.
\end{lemma}
\subsection{A priori estimates} In this subsection, we provide the a priori estimates for the global existence of strong solutions to the system \eqref{main_eq2}. This together with the standard continuity argument yields that the local-in-time strong solution can be extended to the global one.

For notational simplicity, we denote by $U := (h,u)$, i.e., $U_0 = (h_0,u_0)$ and $\|U\|_{H^s}= \|(h,u)\|_{H^s}$ in the rest of this section. We then begin by estimating $H^s$-norm of $U$ in the lemma below. Note that the $L^2$-estimate essentially obtained in Lemma \ref{lem_energy_hu} is even needed for the estimates of higher order derivative of $U$.

\begin{lemma}\label{lem:aux}Let $s > d/2 + 1$ and $T>0$ be given. Suppose $\sup_{0 \leq t \leq T}\|U(\cdot,t)\|_{H^s}^2 \leq \epsilon_1$ for sufficiently small $\epsilon_1 > 0$. Then we have
\[
\sup_{0 \leq t \leq T}\|U(\cdot,t)\|_{L^2} \leq C\|U_0\|_{L^2},
\]
and
\bq\label{est_hk}
\frac{d}{dt}\|\nabla^k U\|_{L^2}^2 + \psi_m\|\nabla^k u\|_{L^2}^2 \leq C\epsilon_1\|\nabla^k U\|_{L^2}^2 + C\|u\|_{H^{k-1}}^2 + C\|U_0\|_{L^2}^2,
\eq
for $1 \leq k \leq s+1$. Here $C>0$ is independent of $t$.
\end{lemma}
\begin{proof} {\bf $L^2$-estimate:} Choosing $\epsilon_1 > 0$ small enough such that
\[
\sup_{0 \leq t \leq T}\|e^{h(\cdot,t)} - 1\|_{L^\infty} \leq \frac12,
\]
we find from Lemma \ref{lem_energy_hu} that
$$\begin{aligned}
&\int_{\T^d} |u(x,t)|^2 \,dx + \int_{\T^d} h(x,t)^2 \,dx + \int_0^t\int_{\T^d \times \T^d} \psi(x-y)|u(x,s)-u(y,s)|^2 e^{h(x,s)+h(y,s)} \,dxdyds\cr
&\qquad \ls \int_{\T^d} |u_0(x)|^2 \,dx + \int_{\T^d} h_0(x)^2 \,dx,
\end{aligned}$$
and, in particular, the following inequality holds:
\bq\label{est_l2}
\sup_{0 \leq t \leq T} \|U(\cdot,t)\|_{L^2} \leq \|U_0\|_{L^2}.
\eq
{\bf $H^k$-estimates for $1 \leq k \leq s+1$:} Taking $\nabla^k$ to the system \eqref{main_eq2}, we obtain
$$\begin{aligned}
\frac12\frac{d}{dt}\int_{\T^d} |\nabla^k U|^2\,dx &= - \int_{\T^d} \nabla^k h \cdot \lt( \nabla^k(\nabla h \cdot u) + \nabla^k(\nabla \cdot u)\rt)dx - \int_{\T^d} \nabla^k u \cdot \lt(\nabla^k(u \cdot \nabla u) + \nabla^{k+1} h \rt)dx\cr
&\quad - \int_{\T^d} \lt(\nabla^k u \cdot \nabla^k\int_{\T^d} \psi(x-y) (u(x) - u(y)
)e^{h(y)}\,dy \rt)dx\cr
&=\frac12 \int_{\T^d} (\nabla \cdot u)|\nabla^k U|^2\,dx - \int_{\T^d} \nabla^k h \cdot [\nabla^k, u\cdot\nabla] h\,dx - \int_{\T^d} \nabla^k u \cdot [\nabla^k, u \cdot \nabla]u\,dx\cr
&\quad - \int_{\T^d} \lt(\nabla^k u \cdot \nabla^k\int_{\T^d} \psi(x-y) (u(x) - u(y)
)e^{h(y)}\,dy \rt)dx\cr
&=: \sum_{i=1}^4 I_i,
\end{aligned}$$
where $[\cdot,\cdot]$ denotes the commutator operator, i.e., $[A,B] = AB - BA$. Using Lemma \ref{lem:bern}, we estimate $I_i,i=1,2,3$ as follows.
$$\begin{aligned}
I_1 &\leq \|\nabla u\|_{L^\infty}\|\nabla^k U\|_{L^2}^2 \leq C\epsilon_1\|\nabla^k U\|_{L^2}^2,\cr
I_2 &\ls \|\nabla^k h\|_{L^2}\lt(\|\nabla^k u\|_{L^2}\|\nabla h\|_{L^\infty} + \|\nabla^k h\|_{L^2}\|\nabla u\|_{L^\infty} \rt) \ls \|\nabla U\|_{L^\infty}\|\nabla^k U\|_{L^2}^2,\cr
I_3 &\ls \|\nabla u\|_{L^\infty}\|\nabla^k u\|_{L^2}^2 \leq C\epsilon_1\|\nabla^k u\|_{L^2}^2.
\end{aligned}$$
For the estimate of $I_4$, we split it into three terms:
$$\begin{aligned}
I_4 &= - \sum_{1 \leq \ell \leq k-1}\binom{k-1}{\ell}\int_{\T^d \times \T^d} \nabla^\ell \psi(x-y) \nabla^{k - \ell}u(x)\cdot \nabla^k u(x)e^{h(y)}\,dydx \,(1 - \delta_{k,1})\cr
&\quad - \int_{\T^d \times \T^d} \nabla^k\psi(x-y) (u(x) - u(y))\cdot \nabla^k u(x) e^{h(y)}\,dydx\cr
&\quad - \int_{\T^d \times \T^d} \psi(x-y)|\nabla^k u|^2 e^{h(y)}\,dydx\cr
&=: I_4^1 + I_4^2 + I_4^3,
\end{aligned}$$
where $\delta_{i,j}$ denotes the Kronecker delta, i.e., $\delta_{i,j} = 1$ if $i=j$ and $\delta_{i,j} = 0$ otherwise. Here $I_4^3$ can be easily bounded from above by
\[
I_4^3 \leq - \psi_m\|\nabla^k u\|_{L^2}^2,
\]
due to \eqref{ass_psi}.
Using the $L^2$-estimate \eqref{est_l2}, we estimate $I_4^2$ as
$$\begin{aligned}
I_4^2 &\leq \|\nabla^k \psi\|_{L^\infty}\int_{\T^d \times \T^d} \lt(|u(x)| + |u(y)| \rt)|\nabla^k u(x)| e^{h(y)}\,dydx\cr
&\leq \|\nabla^k \psi\|_{L^\infty}\lt(\lt(\epsilon_2\int_{\T^d \times \T^d} |\nabla^k u(x)|^2 e^{h(y)}\,dydx \rt) + \frac{1}{4\epsilon_2}\lt(\int_{\T^d \times \T^d} \lt(|u(x)|^2 + |u(y)|^2 \rt)e^{h(y)}\,dydx\rt) \rt)\cr
&\leq C\epsilon_2 \|\nabla^k u\|_{L^2}^2 + \frac{C}{\epsilon_2} \|U_0\|_{L^2}^2,
\end{aligned}$$
where $\epsilon_2 > 0$ will be appropriately determined later. We finally estimate $I_4^1$ as 
\[
I_4^1 \ls \sum_{1 \leq \ell \leq k}\|\nabla^\ell \psi\|_{L^\infty}\|\nabla^{k-\ell} u\|_{L^2}\|\nabla^k u\|_{L^2}\,(1 - \delta_{k,1}) \ls \|u\|_{H^{k-1}}\|\nabla^k u\|_{L^2}.
\]
Collecting all the above estimates yields
\[
\frac{d}{dt}\|\nabla^k U\|_{L^2}^2 + 2\psi_m\|\nabla^k u\|_{L^2}^2 \leq C\epsilon_1\|\nabla^k U\|_{L^2}^2 + C\epsilon_2 \|\nabla^k u\|_{L^2}^2 + C\|u\|_{H^{k-1}}\|\nabla^k u\|_{L^2} + \frac{C}{\epsilon_2}\|U_0\|_{L^2}^2.
\]
We now choose $\epsilon_2 > 0$ small enough such that $2C\epsilon_2 < \psi_m$ and use the Young's inequality to get 
\[
C\|u\|_{H^{k-1}}\|\nabla^k u\|_{L^2} \leq C\|u\|_{H^{k-1}}^2 + \frac{\psi_m}{2}\|\nabla^k u\|_{L^2}^2.
\] 
Then we have
\[
\frac{d}{dt}\|\nabla^k U\|_{L^2}^2 + \psi_m\|\nabla^k u\|_{L^2}^2 \leq C\epsilon_1\|\nabla^k U\|_{L^2}^2 + C\|u\|_{H^{k-1}}^2 + C\|U_0\|_{L^2}^2.
\]
This completes the proof.
\end{proof}
\begin{remark}\label{rmk:small}Compared to \cite{STW}, where the global existence of strong solutions for the compressible Euler equations with a linear damping is studied, we have additional two terms on the right hand side of \eqref{est_hk} without small coefficients. They appear due to the nonlocal nonlinear interaction term and it seems impossible to remove them.
\end{remark}
Similarly as in \cite{STW}, we only get the dissipation rate for the velocity $u$ in the $H^s$-estimate of $U$. As briefly mentioned in Introduction, in order to have the dissipation rate for the density $\rho$, one can estimate $\|\pa_t U\|_{H^{s-1}}$. However, in our case, this strategy will produce rather complex terms because of the interaction term. Thus we employ the idea used in \cite{Choi16,CK16} of using the estimate of a crossing term to get the appropriate dissipation rate for $\rho$. 
\begin{lemma}\label{lem:aux2}Let $s > d/2 + 1$ and $T>0$ be given. Suppose $\sup_{0 \leq t \leq T}\|U(\cdot,t)\|_{H^s}^2 \leq \epsilon_1$ for sufficiently small $\epsilon_1 > 0$. Then, for $1 \leq k \leq s+1$, we have
$$\begin{aligned}
&\frac{d}{dt}\int_{\T^d} \lt(\nabla^{k-1} u \cdot \nabla^k h\rt) dx + \frac12\|\nabla^k h\|_{L^2}^2 \cr
&\quad \leq C\epsilon_1\|\nabla^k U\|_{L^2}^2 + C\epsilon_1 \|\nabla^{k-1}u\|_{L^2}^2 + \|\nabla^k u\|_{L^2}^2 + C\|\nabla u\|_{H^{k-2}}^2(1 - \delta_{k,1}) + C\|U_0\|_{L^2}^2,
\end{aligned}$$
where $C>0$ is independent of $t$.
\end{lemma}
\begin{proof} For $1 \leq k \leq s+1$, a straightforward computation gives
\[
\frac{d}{dt}\int_{\T^d} \lt(\nabla^{k-1} u \cdot \nabla^k h\rt) dx = \int_{\T^d} \lt(\nabla^{k-1} \pa_t u \cdot \nabla^k h\rt) dx + \int_{\T^d} \lt(\nabla^{k-1} u \cdot \nabla^k \pa_t h\rt) dx =: J_1 + J_2,
\]
where $J_2$ can be estimated as 
$$\begin{aligned}
J_2 &= \int_{\T^d} \nabla^k u \cdot \lt(\nabla^{k-1}\lt( \nabla h \cdot u + \nabla \cdot u \rt) \rt)dx\cr
&\leq \|\nabla^k u\|_{L^2}\lt(\|\nabla h\|_{L^\infty}\|\nabla^{k-1} u\|_{L^2} + \|\nabla^k h\|_{L^2}\|u\|_{L^\infty} \rt) + \|\nabla^k u\|_{L^2}^2\cr
&\leq C\epsilon_1\|\nabla^k U\|_{L^2}^2 + C\epsilon_1\|\nabla^{k-1}u\|_{L^2}^2 + \|\nabla^k u\|_{L^2}^2.
\end{aligned}$$
For the estimate of $J_1$, we split it into three terms:
\[
J_1 = - \int_{\T^d} \nabla^k h \cdot \lt(\nabla^{k-1}\lt(u \cdot \nabla u + \nabla h + \int_{\T^d} \psi(x-y)(u(x) - u(y))\rho(y)\,dy \rt) \rt)dx =: J_1^1 + J_1^2 + J_1^3.
\]
Here $J_1^2$ is simply $J_1^2 = -\|\nabla^k h\|_{L^2}^2$ and $J_1^1$ can be easily estimated as 
\[
J_1^1 \leq \|\nabla^k h\|_{L^2}\lt(\|\nabla^{k-1} u\|_{L^2}\|\nabla u\|_{L^\infty} + \|u\|_{L^\infty}\|\nabla^k u\|_{L^2} \rt) \leq C\epsilon_1\|\nabla^k U\|_{L^2}^2 + C\epsilon_1\|\nabla^{k-1}u\|_{L^2}^2.
\]
We next use the similar argument as in previous lemma to estimate $J_1^3$ as 
$$\begin{aligned}
J_1^3 &= \int_{\T^d \times \T^d} \nabla^k h \cdot \nabla^{k-1}\psi(x-y) (u(x) - u(y))\rho(y)\,dydx\cr
&\quad + \sum_{0 \leq \ell < k-1}\binom{k-1}{\ell} \int_{\T^d \times \T^d} \nabla^k h(x)\psi^\ell(x-y)\cdot \nabla^{k-1-\ell} u(x) \rho(y)\,dydx\,(1 - \delta_{k,1})\cr
&\leq C\|\nabla^k h\|_{L^2}\|U_0\|_{L^2} + C\|\nabla^k h\|_{L^2}\|\nabla u\|_{H^{k-2}}(1 - \delta_{k,1})\cr
&\leq \frac12\|\nabla^k h\|_{L^2}^2 + C\|\nabla u\|_{H^{k-2}}^2(1 - \delta_{k,1}) + C\|U_0\|_{L^2}^2.
\end{aligned}$$
Combining all the above estimates, we conclude the desired result.
\end{proof}
In order to handle the crossing term in Lemma \ref{lem:aux2}, we provide the estimate of $\|\nabla^{k-1}u\|_{L^2}$ in the lemma below.
\begin{lemma} \label{lem:aux3}Let $s > d/2 + 1$ and $T>0$ be given. Suppose $\sup_{0 \leq t \leq T}\|U(\cdot,t)\|_{H^s}^2 \leq \epsilon_1$ for sufficiently small $\epsilon_1 > 0$. Then, for $1 \leq k \leq s+1$, we have
\[
\frac{d}{dt}\int_{\T^d} |\nabla^{k-1}u|^2\,dx + \psi_m\|\nabla^{k-1} u\|_{L^2}^2 \leq C\|\nabla^{k-1} u\|_{L^2}^2 + \frac14\|\nabla^k h\|_{L^2}^2 + C\|U_0\|_{L^2}^2,
\]
where $C>0$ is independent of $t$.
\end{lemma}
\begin{proof}Using the similar argument as in Lemma \ref{lem:aux}, we find
$$\begin{aligned}
\frac12\frac{d}{dt}\int_{\T^d} |\nabla^{k-1}u|^2\,dx &= -\int_{\T^d} \nabla^{k-1} u \cdot \lt(\nabla^{k-1}\lt(u \cdot \nabla u + \nabla h + \int_{\T^d} \psi(x-y)(u(x) - u(y))\rho(y)\,dy \rt) \rt)dx\cr
&\leq \|\nabla u\|_{L^\infty}\|\nabla^{k-1} u\|_{L^2}^2 + \|\nabla^{k-1}u\|_{L^2}\|\nabla^k h\|_{L^2} - \frac{\psi_m}{2}\|\nabla^{k-1} u\|_{L^2}^2 + C\|U_0\|_{L^2}^2\cr
&\leq C\|\nabla^{k-1} u\|_{L^2}^2 + \frac14\|\nabla^k h\|_{L^2}^2 - \frac{\psi_m}{2}\|\nabla^{k-1} u\|_{L^2}^2 + C\|U_0\|_{L^2}^2.
\end{aligned}$$
This completes the proof.
\end{proof}
\begin{remark}\label{rmk_rel} Note that 
\[
\int_{\T^d} |\nabla^k U|^2\,dx + \int_{\T^d} \nabla^{k-1} u \cdot \nabla^k h\,dx + \int_{\T^d} |\nabla^{k-1} u|^2\,dx \simeq \|\nabla^k h\|_{L^2}^2 + \|\nabla^{k-1} u\|_{H^1}^2,
\]
in the sense that there exists $c_0 > 0$ such that
$$\begin{aligned}
c_0^{-1}\lt(\|\nabla^k h\|_{L^2}^2 + \|\nabla^{k-1} u\|_{H^1}^2 \rt) &\leq \int_{\T^d} |\nabla^k U|^2\,dx + \int_{\T^d} \nabla^{k-1} u \cdot \nabla^k h\,dx + \int_{\T^d} |\nabla^{k-1} u|^2\,dx\cr
&\leq c_0\lt(\|\nabla^k h\|_{L^2}^2 + \|\nabla^{k-1} u\|_{H^1}^2 \rt).
\end{aligned}$$
\end{remark}
%
%
%
%
\subsection{Proof of Theorem \ref{thm:ext}: global-in-time existence}
In this subsection, we provide the details of the proof of Theorem \ref{thm:ext}. Note that there are some terms on the right hand side of the estimate of $\|U\|_{H^s}$ whose coefficients cannot be small enough, as mentioned in Remark \ref{rmk:small}. Thus we need to combine the estimates in Lemmas \ref{lem:aux}, \ref{lem:aux2}, and \ref{lem:aux3} carefully to have a Gronwall type inequality for $\|U\|_{H^s}$.    
\begin{proposition}\label{prop:en} Let $s > d/2 + 1$ and $T>0$ be given. Suppose $\sup_{0 \leq t \leq T}\|U(\cdot,t)\|_{H^s}^2 \leq \epsilon_1$ for sufficiently small $\epsilon_1 > 0$. Then we have
\bq\label{est_glo}
\sup_{0 \leq t \leq T}\|U(\cdot,t)\|_{H^s} \leq C\|U_0\|_{H^s},
\eq
where $C$ is a positive constant independent of $t$.
\end{proposition}
\begin{proof}For the proof, we use the induction argument in $s$. It follows from Lemma \ref{lem:aux} that the inequality \eqref{est_glo} holds for $s = 0$. Let us assume that \eqref{est_glo} holds for any $m < s$. Then we now combine the inequalities in Lemmas \ref{lem:aux}, \ref{lem:aux2}, and \ref{lem:aux3} with $k=m+1 \leq s$ to find
$$\begin{aligned}
&\frac{d}{dt}\lt(\int_{\T^d} |\nabla^{m+1} U|^2\,dx + \delta_0 \int_{\T^d} \lt(\nabla^m u \cdot \nabla^{m+1}h\rt) dx + \delta_1 \int_{\T^d} |\nabla^m u|^2\,dx \rt) \cr
&\quad \leq -\psi_m\|\nabla^{m+1} u\|_{L^2}^2  + C\epsilon_1\|\nabla^{m+1} U\|_{L^2}^2 + C\|u\|_{H^m}^2 + C\|U_0\|_{L^2}^2\cr
&\qquad - \frac{\delta_0}{2}\|\nabla^{m+1} h\|_{L^2}^2 + C\epsilon_1\delta_0 \|\nabla^m u\|_{L^2}^2 + \delta_0\|\nabla^{m+1} u\|_{L^2}^2 + C\delta_0\|\nabla u\|_{H^{m-1}}^2(1 - \delta_{m,0})\cr
&\qquad - \psi_m\delta_1\|\nabla^m u\|_{L^2}^2 + C\delta_1\|\nabla^m u\|_{L^2}^2 + \frac{\delta_1}{4}\|\nabla^{m+1} h\|_{L^2}^2\cr
&\quad \leq -(\psi_m - \delta_0)\|\nabla^{m+1} u\|_{L^2}^2 - \frac12\lt(\delta_0 - \frac{\delta_1}{2} \rt)\|\nabla^{m+1}h\|_{L^2}^2   - \psi_m\delta_1\|\nabla^m u\|_{L^2}^2 \cr
&\qquad  + C\epsilon_1\|\nabla^{m+1} U\|_{L^2}^2 + C\|U_0\|_{H^m}^2,
\end{aligned}$$
where we used the fact that the inequality \eqref{est_glo} holds for $m$. Then choosing $\delta_0, \delta_1 > 0$ such that $\psi_m > \delta_0 > \delta_1/2$ gives that the right hand side can be bounded from above by
\[
-\lt(\min\lt\{\psi_m - \delta_0, \frac12\lt(\delta_0 - \frac{\delta_1}{2} \rt) \rt\} - C\epsilon_1\rt)\|\nabla^{m+1} U\|_{L^2}^2 - \psi_m\delta_1\|\nabla^m u\|_{L^2}^2   + C\|U_0\|_{H^m}^2.
\]
This together with the relation in Remark \ref{rmk_rel} concludes that the inequality \eqref{est_glo} holds for $m+1$, and this completes the proof.
\end{proof}
We are now in a position to prove Theorem \ref{thm:ext}.

\begin{proof}[Proof of Theorem \ref{thm:ext}] The local existence is obtained in Lemma \ref{lem:local}. Choose a positive constant 
\[
M := \min \{M_1,\epsilon_1^2\},
\]
where $M_1$ and $\epsilon_1$ are given in Lemma \ref{lem:local} and Proposition \ref{prop:en}, respectively. Furthermore, we choose the initial data such that
\bq\label{ini-con}
\|U_0\|_{H^{s+1}}^2 \leq \frac{M}{2(1+C)},
\eq
where $C$ is a positive constant appeared in Proposition \ref{prop:en}. Let us define the lifespan of the solutions for the system \eqref{main_eq2}-\eqref{ini-main_eq2} as
\[
S := \lt\{ t \geq 0 : \sup_{0 \leq s \leq t} \|U(\cdot,s)\|_{H^{s+1}}^2 \leq M \rt\}.
\]
Since the initial data satisfy \eqref{ini-con}, $S \neq \phi$. Suppose $\mathcal{T}:=\sup S$ is finite, then we find
\[
M=\sup_{0\leq t\leq \mathcal{T}}\|U(\cdot,t)\|_{H^{s+1}}^2 \leq C\|U_0\|_{H^{s+1}}^2 \leq \frac{CM}{2(1+C)} \leq \frac M2 < M,
\]
due to $\|U(\cdot,t)\|_{H^{s+1}}^2 \leq \epsilon_1^2$ for $t \in [0,\mathcal{T}]$. This is a contradiction, hence $\mathcal{T} = \infty$, and this completes the proof.
\end{proof}
%
%
%
%
\section{Large-time behavior}\label{sec_lt}
In this part, we study the large-time behavior of global classical solutions to the system \eqref{CEN}-\eqref{ini-CEN}. 

For this, we first introduce a type of temporary Lyapunov function $\mathcal{E}(t)$ and its corresponding dissipation $\mathcal{D}(t)$:
\[
\mathcal{E}(t) := \frac12\int_{\T^d} \rho|u - m_c|^2 \,dx + \int_{\T^d} \rho \int_1^\rho \frac{z - 1}{z^2} \,dzdx,
\]
and
\[
\mathcal{D}(t) := \frac12\int_{\T^d \times \T^d} \psi(x-y)|u(x)-u(y)|^2 \rho(x)\rho(y) \,dx dy.
\]
\begin{lemma}\label{lem_ef1}Let $(\rho ,u)$ be a global classical solution to the system \eqref{CEN}-\eqref{ini-CEN}. Then we have
\[
\frac{d}{dt} \mathcal{E}(t) + \mathcal{D}(t) = 0,
\]
for $t \geq 0$.
\end{lemma}
\begin{proof} A straightforward computation together with using the symmetry assumption on $\psi$ gives
\begin{align*}
\begin{aligned}
\frac{d}{dt}\lt(\frac12\int_{\T^d} \rho|u-m_c|^2 dx + \int_{\T^d} \rho\ln\rho\, dx \rt)
&= -\int_{\T^d} \psi(x-y)(u(x) - u(y)) \cdot (u(x)-m_c) \rho(x)\rho(y) dxdy\cr
&= -\frac12\int_{\T^d \times \T^d} \psi(x-y)|u(x)-u(y)|^2 \rho(x)\rho(y) dx dy.
\end{aligned}
\end{align*}
We now use the equality in Lemma \ref{lem:2} to complete the proof.
\end{proof}
It follows from Lemma \ref{lem_ef1} that 
\[
\frac{d}{dt}\mathcal{E}(t) + \psi_m\int_{\T^d} \rho|u - m_c|^2 \,dx \leq 0,
\]
where we used the lower bound assumption on $\psi$ \eqref{ass_psi}. However, the above estimate is not enough to provide the desired large-time behavior estimate. To handle this problem, we use the Bogovskii operator introduced in Section \ref{sec_bogo} and introduce modified temporary Lyapunov function $\mathcal{E}^\sigma$ and its dissipation $D^\sigma$ as follows:
\[
\mathcal{E}^\sigma(t) := \mathcal{E}(t) + \sigma\int_{\T^d}\rho(u-m_c)\mathcal{B}[\rho - 1] \,dx,
\]
and
\begin{align*}
\begin{aligned}
\mathcal{D}^\sigma(t) &:= \mathcal{D}(t) + \sigma\lt( \int_{\T^d} (\rho u \otimes u) : \nabla \mathcal{B}[\rho - 1] + (\rho - 1)^2\, dx\rt)\cr
&\quad -\sigma\lt( \int_{\T^d} \rho(u-m_c) \mathcal{B}[\nabla \cdot (\rho u)] + \pa_t(\rho m_c) \mathcal{B}[\rho - 1]\,dx\rt)\cr
&\quad -\sigma\int_{\T^d \times \T^d} \psi(x-y)(u(x) - u(y)) \cdot \mathcal{B}[\rho - 1] \rho(x)\rho(y)\,dx dy.
\end{aligned}
\end{align*}
In the following two lemmas, we provide that the above newly defined functions $\mathcal{E}^\sigma$ and $\mathcal{D}^\sigma$ are good approximations of our Lyapunov function $\mathcal{F}$ for $\sigma > 0$ small enough.
\begin{lemma} Let $(\rho ,u)$ be a global classical solution to the system \eqref{CEN}-\eqref{ini-CEN}. Suppose $\rho \in [0,\bar\rho]$. Then there exist positive constants $c_3, c_4$, which are independent of $t$, such that 
\[
c_3\mathcal{F}(t) \leq \mathcal{E}^\sigma(t) \leq c_4 \mathcal{F}(t), \quad t \geq 0,
\]
for $\sigma >0$ small enough.
\end{lemma}
\begin{proof}Note that
\[
\sigma\lt|\int_{\T^d}\rho(u-m_c)\mathcal{B}[\rho - 1] \,dx\rt| \leq \frac\sigma2\int_{\T^d} \rho|u-m_c|^2 \,dx + \frac{C\sigma \bar\rho}{2}\int_{\T^d} (\rho - 1)^2 \,dx.
\]
Then this and together with the estimate in Lemma \ref{lem:3} yield 
\[
\mathcal{E}^\sigma(t) \ge \frac{1 - \sigma}{2}\int_{\T^d} \rho|u-m_c|^2 \,dx + \lt( c_1 - \frac{C\sigma \bar\rho}{2}\rt)\int_{\T^d} (\rho - 1)^2 \,dx 
\]
By choosing $\sigma > 0$ small enough, we find that there exists a $c_3 > 0$ such that $c_3\mathcal{F}(t) \leq \mathcal{E}^\sigma(t)$ for all $t \geq 0$. The estimate for upper bound of $\mathcal{E}^\sigma(t)$ is clearly obtained. 
\end{proof}
\begin{lemma}Let $(\rho ,u)$ be a global classical solution to the system \eqref{CEN}-\eqref{ini-CEN}. Suppose $\rho \in [0,\bar\rho]$. Then there exist positive constants $c_5,c_6$, which are independent of $t$, such that
\[
c_5 \mathcal{F}(t) \leq \mathcal{D}^\sigma(t) \leq c_6 \mathcal{F}(t), \quad t \geq 0.
\]
for $\sigma > 0$ small enough.
\end{lemma}
\begin{proof} We estimate the each terms in $\mathcal{D}^\sigma(t) - \mathcal{D}(t)$. We set
\begin{align*}
\begin{aligned}
\sum_{i=1}^5 I_i &:= \sigma \int_{\T^d} (\rho u \otimes u) : \nabla \mathcal{B}[\rho - 1] + (\rho - 1)^2 \,dx  \cr
&\quad -\sigma \int_{\T^d} \rho(u-m_c) \cdot \mathcal{B}[\nabla \cdot (\rho u)] + \pa_t(\rho m_c) \cdot \mathcal{B}[\rho - 1]\,dx\cr
&\quad -\sigma\int_{\T^d \times \T^d} \psi(x-y)(u(x) - u(y)) \cdot \mathcal{B}[\rho - 1] \rho(x)\rho(y) \,dx dy.
\end{aligned}
\end{align*}
By adding and subtracting, we first find that
\begin{align*}
\begin{aligned}
\sigma\lt|\int_{\T^d} (\rho u \otimes u) : \nabla \mathcal{B}[\rho - 1] \,dx\rt| &= \sigma \int_{\T^d} \rho (( u - m_c) \otimes u) : \nabla
\mathcal{B}[\rho - 1]\,dx \cr
&\quad + \sigma \int_{\T^d} \rho  (m_c \otimes ( u -
m_c)) : \nabla \mathcal{B}[\rho - 1]\,dx \cr
&\quad  + \sigma \int_{\T^d} (\rho - \rho_c) (m_c \otimes m_c) : \nabla \mathcal{B}[\rho - 1]\,dx \cr
 & =: I_1^1 + I_1^2 +  I_1^3.
 \end{aligned}
\end{align*}
Here $I_1^i,i=1,2,3$ are estimated as follows.
\begin{align*}
\begin{aligned}
I_1^1 &\leq \frac{\sigma^{1/2}\bar \rho \|u\|_{L^\infty}}{2}
\int_{\T^d} \rho| u  - m_c|^2\,dx + C\sigma^{3/2}\int_{\T^d} (\rho - 1)^2 \,dx, \cr
I_1^2 &\leq \frac{\sigma^{1/2}\bar \rho |m_c(0)|^2}{2}
\int_{\T^d} \rho| u  - m_c|^2\,dx + C\sigma^{3/2}\int_{\T^d} (\rho - 1)^2 \,dx, \cr
I_1^3 &\leq C\sigma |m_c(0)|^2\int_{\T^d} (\rho - 1)^2 \,dx,
\end{aligned}
\end{align*}
due to the conservation of momentum. Thus we have
\[
I_1 \leq \lt(\frac{\sigma^{1/2}\bar \rho(\|u\|_{L^\infty} + |m_c(0)|^2)}{2} \rt)\int_{\T^d} \rho|u-m_c|^2 \,dx + C\sigma\lt(|m_c(0)|^2 + \sigma^{1/2}\rt)\int_{\T^d}(\rho-1)^2 \,dx.
\]
For the estimate of $I_3$, we deduce that
\begin{align*}
\begin{aligned}
&\sigma\int_{\T^d} \rho (u-m_c)\cdot\mathcal{B}[\nabla\cdot(\rho u)]\,dx \cr
& \hspace{.5cm} = \sigma\int_{\T^d} \rho (u - m_c)\cdot\mathcal{B}[\nabla \cdot (\rho (u -
m_c))]\,dx + \sigma\int_{\T^d} \rho (u - m_c)\cdot\mathcal{B}[\nabla \cdot (\rho
m_c)]\,dx \cr
& \hspace{.5cm} = \sigma\int_{\T^d} \rho (u - m_c)\cdot\mathcal{B}[\nabla \cdot
(\rho (u - m_c))]\,dx + \sigma\int_{\T^d} \rho (u - m_c)\cdot\mathcal{B}[\nabla
\cdot ((\rho - 1) m_c)]\,dx \cr
&\hspace{.5cm}\leq C \sigma\bar\rho\int_{\T^d} \rho|u-m_c|^2\,dx
+ \sigma\int_{\T^d}\rho|u-m_c|| \mathcal{B}[\nabla
\cdot ((\rho - 1) m_c)]| \,dx \cr
&\hspace{.5cm} \leq C \lt(\sigma + \sigma^{1/2}\rt)\bar\rho\int_{\T^d} \rho|u-m_c|^2\,dx + C \sigma^{3/2}|m_c(0)|^2\int_{\T^d} (\rho - 1)^2 \,dx. \cr
\end{aligned}
\end{align*}
We next estimate $I_4$. Since
\[
\partial_t(\rho m_c) = m_c \partial_t \rho + \rho m_c^{\prime}
= - m_c \nabla_x \cdot (\rho u),
\]
we obtain
\begin{align*}
\begin{aligned}
\sigma \int_{\T^d} \partial_t(\rho m_c) \cdot \mathcal{B}[\rho - 1] \, dx  &= \sigma \int_{\T^d} \rho ( u - m_c) \cdot m_c \nabla \mathcal{B}[\rho - 1] \,dx + \sigma \int_{\T^d} (\rho - \rho_c)|m_c|^2 \nabla \mathcal{B} [\rho - 1] \,dx\cr
&:=I_4^1 + I_4^2.
\end{aligned}
\end{align*}
Here, the terms $I_4^i,i=1,2$ can be estimated as
\begin{align*}
\begin{aligned}
I_4^1 &\leq \sigma^{1/2}\int_{\T^d} \rho|u-m_c|^2 \,dx + C\sigma^{3/2}|m_c(0)|^2\bar\rho\int_{\T^d} ( \rho - 1)^2 \,dx,\cr
I_4^2 &\leq C\sigma |m_c(0)|^2 \int_{\T^d} ( \rho - 1)^2 \,dx.\cr
\end{aligned}
\end{align*}
This yields
\[
I_4 \leq \sigma^{1/2}\int_{\T^d} \rho|u-m_c|^2 \,dx + C\sigma\lt(\sigma^{1/2} |m_c(0)|^2 \bar\rho + E_0 \rt)\int_{\T^d} (\rho - 1)^2\,dx.
\]
Finally we estimate $I_5$ as follows.
$$\begin{aligned}
I_5 &= \sigma\int_{\T^d \times \T^d} \psi(x-y)(u(x) - m_c) \cdot \mathcal{B}[\rho(x) - 1] \rho(x)\rho(y) \,dx dy\cr
&\quad + \sigma\int_{\T^d \times \T^d} \psi(x-y)(m_c - u(y))\cdot \mathcal{B}[\rho(x) - 1] \rho(x) \rho(y) \,dx dy\cr
&\leq C\sigma^{1/2}\|\psi\|_{L^\infty}\lt( \bar\rho\int_{\T^d} \rho|u-m_c|^2\,dx + \sigma\int_{\T^d} (\rho - 1)^2 \,dx\rt)\cr
&\quad + C\sigma^{1/2}\|\psi\|_{L^\infty}\lt(\int_{\T^d} \rho |u-m_c|^2\,dx + \sigma\int_{\T^d} (\rho - 1)^2 \,dx \rt),
\end{aligned}$$
where we used the following estimate:
$$\begin{aligned}
&\sigma\lt(\int_{\T^d} \rho|u-m_c| \,dx \rt)\lt(\int_{\T^d} \rho|\mathcal{B}[\rho - 1]|\,dx \rt)\cr
&\quad \leq C\sigma^{1/2}\lt(\int_{\T^d} \rho|u-m_c| \,dx \rt)^2 + C\sigma^{3/2}\lt(\int_{\T^d} \rho|\mathcal{B}[\rho - 1]|\,dx \rt)^2\cr
&\quad \leq C\sigma^{1/2}\lt(\int_{\T^d} \rho |u-m_c|^2 \,dx + \sigma\int_{\T^d} (\rho - 1)^2 \,dx \rt).
\end{aligned}$$
We now combine all the above estimates to find
$$\begin{aligned}
\mathcal{D}^\sigma(t) &\geq \frac12\int_{\T^d \times \T^d} \psi(x-y)|u(x)-u(y)|^2 \rho(x)\rho(y) \,dx dy + C_1\int_{\T^d}(\rho - 1)^2 \,dx\cr
&\quad  - C_2 \sigma^{1/2}\int_{\T^d} \rho|u-m_c|^2 \,dx,
\end{aligned}$$
where $C_1$ and $C_2$ are given by
\[
C_1:= \sigma\lt(1 - C|m_c(0)|^2\lt(1+\sigma^{1/2}(1+\bar\rho)\rt) - C\sigma^{1/2}\lt( 1 + \|\psi\|_{L^\infty}\rt)\rt),
\]
and
\[
C_2:=C\bar\rho\lt( 1 + \frac{1}{\bar\rho}+\frac{\|u\|_{L^\infty} + |m_c(0)|^2}{2} + \frac{\|\psi\|_{L^\infty}}{\bar\rho}\rt).
\]
We notice that $C_1$ is positive constant for sufficiently small $\sigma$ and $|m_c(0)|$. We also find that
$$\begin{aligned}
\frac12\int_{\T^d \times \T^d} \psi(x-y)|u(x) - u(y)|^2 \rho(x) \rho(y) \,dx dy & \geq \frac{\psi_m}{2}\int_{\T^d \times \T^d}|u(x) - u(y)|^2 \rho(x) \rho(y) \,dx dy\cr
&= \psi_m\int_{\T^d}\rho|u-m_c|^2\,dx,
\end{aligned}$$
where we used
\[
\int_{\T^d \times \T^d} (u(x) - m_c) \cdot (m_c - u(y)) \rho(x) \rho(y) \,dx dy = 0.
\]
Thus we have
\begin{align*}
\begin{aligned}
\mathcal{D}^\sigma(t) &\geq \lt(\psi_m - C_2 \sigma^{1/2}\rt)\int_{\T^d} \rho|u-m_c|^2 \,dx + C_1\int_{\T^d} (\rho - 1)^2 \,dx
\end{aligned}
\end{align*}
We again select $\sigma$ and $|m_c(0)|$ small enough to get the positive $C_1>0$ and $\psi_m  - C_2\sigma^{1/2} > 0$. This yields that the all coefficients above are positive. The estimate for upper bound of $\mathcal{D}^\sigma$ is clear. Hence we have
\[
\frac{d}{dt}\mathcal{E}^\sigma(t) + \frac{c_5}{c_4}\mathcal{E}^\sigma(t) \leq 0\leq \frac{d}{dt}\mathcal{E}^\sigma(t) + \frac{c_6}{c_3}\mathcal{E}^\sigma(t), \quad t \ge 0 ,
\]
due to 
\[
\frac{c_6}{c_3}\mathcal{E}^\sigma(t) \geq c_6\mathcal{F}(t) \geq \mathcal{D}^\sigma(t) \geq c_5\mathcal{F}(t) \geq \frac{c_5}{c_4}\mathcal{E}^\sigma(t),
\] 
for $\sigma > 0$ small enough, and this deduces
\[
\frac{c_3}{c_4}\mathcal{F}_0e^{-\frac{c_6}{c_3}t} \leq \frac{1}{c_4}\mathcal{E}_0^\sigma e^{-\frac{c_6}{c_3}t} \leq \frac{1}{c_4}\mathcal{E}^\sigma(t) \leq \mathcal{F}(t) \leq \frac{1}{c_3}\mathcal{E}^\sigma(t) \leq \frac{1}{c_3}\mathcal{E}_0^\sigma e^{-\frac{c_5}{c_4}t} \leq \frac{c_4}{c_3}\mathcal{F}_0e^{-\frac{c_5}{c_4}t},
\]
for $t \geq 0$ and $\sigma > 0$ small enough. This concludes our desired result.
\end{proof}

%
%
%
%
\section*{Acknowledgements}
This research was supported by NRF grant(No. 2017R1C1B2012918 and 2017R1A4A1014735) and POSCO Science Fellowship of POSCO TJ Park Foundation.
%
%
%
%


\begin{thebibliography}{10}
\bibitem{BCC11} F. Bolley, J. A. Ca\~nizo, J. A. Carrillo, Stochastic Mean-Field Limit: Non-Lipschitz Forces \& Swarming, Math. Models Methods Appl. Sci., 21, (2011), 2179--2210.
\bibitem{CCR11}  J. A. Ca\~nizo, J. A. Carrillo, and J. Rosado, A well-posedness theory in measures for some kinetic models of collective motion, Math. Models Methods Appl. Sci., 21, (2011), 515--539.
\bibitem{CCTT16} J. A. Carrillo, Y.-P. Choi, E. Tadmor, and C. Tan, Critical thresholds in 1D Euler equations with nonlocal forces, Math. Models Methods Appl. Sci., 26, (2016), 185--206.
\bibitem{CCH14} J. A. Carrillo, Y.-P. Choi, and M. Hauray, The derivation of swarming models: Mean-field limit and Wasserstein distances, Collective Dynamics from Bacteria to Crowds: An Excursion Through Modeling, Analysis and Simulation, Series: CISM International Centre
for Mechanical Sciences, Springer, 533, (2014), 1--45.
\bibitem{CCH14_2} J. A. Carrillo, Y.-P. Choi, and M. Hauray, Local well-posedness of the generalized Cucker-Smale model with singular kernels, ESAIM Proc., 47, (2014), 17--35.
\bibitem{CCHSapp} J. A. Carrillo, Y.-P. Choi, M. Hauray, and S. Salem, Mean-field limit for collective behavior
models with sharp sensitivity regions, to appear in J. Eur. Math. Soc.
\bibitem{CCP17} J. A. Carrillo, Y.-P. Choi, and S. P\'erez, A review on attractive-repulsive hydrodynamics for consensus in collective behavior, Active particles. Vol. 1. Advances in theory, models, and applications, 259--298, Model. Simul. Sci. Eng. Technol., Birkh\"auser/Springer, Cham, 2017. 
\bibitem{CCZ16} J. A. Carrillo, Y.-P. Choi, and E. Zatorska,  On the pressureless damped Euler-Poisson equations with quadratic confinement: critical thresholds and large-time behavior, Math. Models Methods Appl. Sci., 26, (2016), 2311--2340.
\bibitem{CFRT10} J. A. Carrillo, M. Fornasier, J. Rosado, and G. Toscani, Asymptotic Flocking Dynamics for the kinetic Cucker-Smale model, SIAM J. Math. Anal., 42, (201), 218--236.
\bibitem{Chen05}  G.-Q. Chen, Euler equations and related hyperbolic conservation laws, In: Handbook of differential equations, 2, edited
by C. M. Dafermos, E. Feireisl, Amsterdam: Elsevier Science, (2005), 1--104.
\bibitem{Choi15} Y.-P. Choi, Compressible Euler equations interacting with incompressible flow, Kinet. Relat. Models, 8, (2015), 335--358.
\bibitem{Choi16} Y.-P. Choi, Global classical solutions and large-time behavior of the two-phase fluid model, SIAM J. Math. Anal., 48, (2016), 3090--3122.
\bibitem{Choi16_2} Y.-P. Choi, Large-time behavior for the Vlasov/compressible Navier-Stokes equations, J. Math. Phys., 57, (2016), 071501.
\bibitem{CHL17} Y.-P. Choi, S.-Y. Ha, and Z. Li, Emergent dynamics of the Cucker-Smale flocking model and its variants, Active particles. Vol. 1. Advances in theory, models, and applications, 299--331, Model. Simul. Sci. Eng. Technol., Birkh\"auser/Springer, Cham, 2017.
\bibitem{CHpre} Y.-P. Choi and J. Haskovec, Hydrodynamic Cucker-Smale model with normalized communication weights and time delay, preprint.
\bibitem{CK16} Y.-P. Choi and B. Kwon, The Cauchy problem for the pressureless Euler/isentropic Navier-Stokes equations, J. Differential Equations, 261, (2016), 654--711.
\bibitem{CS18} Y.-P. Choi and S. Salem, Propagation of chaos for aggregation equations with no-flux boundary conditions and sharp sensing zones, Math. Models Methods Appl. Sci., 28, (2018), 223--258.
\bibitem{CS07} F. Cucker and S. Smale, Emergent behavior in flocks, IEEE Trans. Automat. Control, 52, (2007), 852--862.
\bibitem{DKRTpre} T. Do, A. Kiselev, L. Ryzhik, and C. Tan, Global regularity for the fractional Euler alignment system, to appear in Arch. Ration. Mech. Anal.
\bibitem{FKpre} A. Figalli and M.-J. Kang, A rigorous derivation from the kinetic Cucker-Smale model to the pressureless Euler system with nonlocal alignment, preprint. 
\bibitem{Galdi} G. P. Galdi, An introduction to the mathematical theory of the Navier-Stokes equations I, Springer-Verlag, New York, 1994.
\bibitem{HKK14} S.-Y. Ha, M.-J. Kang, and B. Kwon, A hydrodynamic model for the interaction of Cucker-Smale particles and incompressible fluid, Math. Models Methods Appl. Sci., 24, (2014), 2311--2359.
\bibitem{HKK15} S.-Y. Ha, M.-J. Kang, and B. Kwon, Emergent dynamics for the hydrodynamic Cucker-Smale system in a moving domain,  SIAM J. Math. Anal., 47, (2015), 3813--3831.
\bibitem{HL09} S.-Y. Ha, J.-G. Liu, A simple proof of the Cucker-Smale flocking dynamics and mean-field limit, Commun. Math. Sci., 7,(2009), 297--325.
\bibitem{HT08} S.-Y. Ha and E. Tadmor, From particle to kinetic and hydrodynamic descriptions of flocking, Kinetic and Related Models,  1, (2008), 415--435.
\bibitem{KMT13} T. Karper, A. Mellet, and K. Trivisa, Existence of weak solutions to kinetic flocking models, SIAM J. Math. Anal., 45, (2013), 215--243.
\bibitem{KMT15} T. Karper, A. Mellet, and K. Trivisa, Hydrodynamic limit of the kinetic Cucker-Smale flocking model, Math. Models Methods Appl. Sci., 25, (2015), 131--163.
\bibitem{KTpre} A. Kiselev and C. Tan, Global regularity for 1D Eulerian dynamics with singular interaction forces, preprint.
\bibitem{Majda} A. Majda, Compressible fluid flow and systems of conservation laws in several space variables, Applied Mathematical Sciences, New York: Springer-Verlag, 1984:53.
\bibitem{ST17} R. Shvydkoy and E. Tadmor, Eulerian dynamics with a commutator forcing, Trans. Mathematics and Applications, 1, (2017), 1--26.
\bibitem{STapp} R. Shvydkoy and E. Tadmor, Eulerian dynamics with a commutator forcing III: Fractional diffusion of order $0<\alpha<1$, to appear in Physica D.
\bibitem{STW} T. C. Sideris, B. Thomases, and D. Wang, Long time behavior of solutions to the 3D compressible Euler equations with damping, Comm. Partial Differential Equations, 28, (2003), 795--816.
\bibitem{TT14} E. Tadmor and C. Tan, Critical thresholds in flocking hydrodynamics with non-local alignment, Philos. Trans. R. Soc. Lond. Ser. A Math. Phys. Eng. Sci., 372, 20130401, (2014).
\end{thebibliography}
\end{document}